\documentclass[12pt]{article}

\usepackage{amsmath, epsfig, cite}
\usepackage{amssymb}
\usepackage{ytableau}
\usepackage{amsfonts}
\usepackage{tikz}
\usepackage{float}
\usepackage{cases}
\usepackage{bm, amscd,  amsfonts, amssymb}
\usepackage{amsmath,amsthm, epsfig}
\newtheorem{thm}{Theorem}[section]

\newtheorem{cor}[thm]{Corollary}

\newtheorem{lem}[thm]{Lemma}

\numberwithin{equation}{section}

\makeatletter \@addtoreset{equation}{section} \makeatother

\usepackage{geometry}

\tikzstyle{every node}=[circle,inner sep=1pt,fill=white!60]
\tikzstyle{tn}=[shape=circle, draw, color=black!70]
\allowdisplaybreaks

\begin{document}
\begin{center}
{\large\bf
Identities on Factorial Grothendieck Polynomials}
 \end{center}

\begin{center}
{\small  Peter L. Guo$^1$ and Sophie C.C. Sun$^2$}

\vskip 3mm
Center for Combinatorics, LPMC\\
Nankai University\\
Tianjin 300071, P.R. China
\\[3mm]

\vskip 4mm

$^1$lguo@nankai.edu.cn,
 $^2$suncongcong@mail.nankai.edu.cn
\end{center}

\begin{abstract}
 Gustafson and Milne proved an  identity on the Schur function
indexed by a partition of the form  $(\lambda_1-n+k,\lambda_2-n+k,\ldots,\lambda_k-n+k)$.
On the other hand,
 Feh\'{e}r, N\'{e}methi and Rim\'{a}nyi found an identity
on the  Schur function
indexed by  a partition of the form $(m-k,\ldots,m-k, \lambda_1,\ldots,\lambda_k)$.
Feh\'{e}r, N\'{e}methi and Rim\'{a}nyi gave a geometric explanation    of
their identity, and they raised the question of finding a combinatorial proof.
In this paper, we establish  a Gustafson-Milne type identity as well
as a  Feh\'{e}r-N\'{e}methi-Rim\'{a}nyi type identity  for
 factorial Grothendieck polynomials.
Specializing a factorial Grothendieck polynomial
to a Schur function, we obtain a combinatorial proof
of the Feh\'{e}r-N\'{e}methi-Rim\'{a}nyi identity.
\end{abstract}

\section{Introduction}

Throughout this paper, we let $n$ be a positive integer.
We shall write  $[n]=\{1,2,\ldots,n\}$ and
use $\binom{[n]}{k}$ to represent the set of $k$-subsets of $[n]$.
For a $k$-subset $S$ of $[n]$, we denote the elements of $S$
by $i_1<i_2<\cdots<i_k$.
For a partition $\lambda=(\lambda_1,\lambda_2,\ldots, \lambda_k)$,
Gustafson and Milne \cite{GM83} proved
the following identity on Schur functions:
\begin{equation}\label{thm milne}
  s_{(\lambda_1-n+k,\lambda_2-n+k,\ldots,\lambda_k-n+k)}(x_1,x_2,\ldots,x_n)
  =\sum_{S\in\binom{[n]}{k}}
  \frac{s_{\lambda}(x_S)}{\prod\limits_{i\in S}\prod\limits_{j\in \overline{S}}(x_i-x_j)},
\end{equation}
where $s_\lambda(x_S)=s_{\lambda}(x_{i_1}, x_{i_2},\ldots, x_{i_k})$ and $\overline{S}=[n]\setminus S$ is  the complement of $S$.
In the case  $\lambda_k-n+k< 0$, the left-hand side of \eqref{thm milne}
is zero.
Gustafson and Milne \cite{GM83} deduced \eqref{thm milne}
based on the determinantal formula of  Schur functions.
Chen and Louck \cite{ChLo96} gave an alternative approach
 by using  an interpolation
formula for symmetric functions.

Taking  $\lambda=(n-1)$ and then replacing $x_i$ with $x_i^{-1}$ for $i\geq 1$,
the Gustafson-Milne  identity
specializes to the Good's identity
\begin{equation}\label{Good}
  1=\sum_{i=1}^n\prod_{j=1\atop{j\not=i}}^n\left(1-\frac{x_i}{x_j}\right)^{-1},
  \end{equation}
which played a crucial  role in the proof of the
Dyson conjecture \cite{Good}.
In the case $\lambda=(m)$, the Gustafson-Milne  identity  becomes an identity
 due to
Louck \cite{Louck60,Louck70}
\begin{equation}\label{thm louck eq}
  h_{m-n+1}(x_1,x_2,\ldots,x_n)=\frac{\sum_{i=1}^nx_i^m}
  {\prod\limits_{i=1}^n\prod\limits_{j=1\atop{j\not=i}}^n(x_i-x_j)},
\end{equation}
where $h_k(x)$ is the  complete homogeneous symmetric function.
See \cite{GM83} for various applications of the Louck's identity.
The Gustafson-Milne
identity  also appeared in  the calculations  of Thom polynomials \cite{FR}.

On the other hand, in the study of equivariant cohomology  classes of matrix matroid varieties,
Feh\'{e}r, N\'{e}methi and
Rim\'{a}nyi \cite[Theorem 7.5]{FNR12} found the following identity:
\begin{equation}\label{FNR-S}
s_{({\small{\underbrace{m-k,m-k,\ldots,m-k}_{n-k}}},\lambda_1,\lambda_2,\ldots,\lambda_k)}(x_1,x_2,\ldots,x_n)
  =\sum_{S\in\binom{[n]}{k}} s_{\lambda}(x_S)\frac{ \prod\limits_{j\in \overline{S}} x_j^m}
{\prod\limits_{i\in S} \prod\limits_{j\in \overline{S}}(x_j-x_i)}.
\end{equation}
Feh\'{e}r, N\'{e}methi and
Rim\'{a}nyi \cite{FNR12} gave a geometric explanation  for the identity
\eqref{FNR-S}, and  posed  the question of finding
a combinatorial proof.

In this paper, we establish  a Gustafson-Milne type identity as well
as a  Feh\'{e}r-N\'{e}methi-Rim\'{a}nyi type identity  for
 factorial Grothendieck polynomials.
The factorial Grothendieck polynomial $G_\lambda(x|y)$ is the double
Grothendieck polynomial corresponding to a Grassmannian permutation.
These polynomials   can also be  interpreted  in terms of set-valued
tableaux  \cite{Buch,KMY,KMY-2,Mc,Mat17}, or   expressed
as the quotient of determinants \cite{HIMN,Hu-Ma,Ik-Na,Mat}. The lowest degree
homogeneous component of $G_\lambda(x|y)$ is equal to the factorial
Schur function $s_\lambda(x|y)$, which is the double Schubert polynomial
of a Grassmannian permutation and has received extensive
  attention, see, for
example,  \cite{Bi-Lo,Bi-Lo-2,Ch-Lo,Go-Gr,KnTao03,Mac2,Mi,Mo-Sa}.
Restricting a factorial Grothendieck
polynomial to a Schur function, we are led to a combinatorial
proof of the identity \eqref{FNR-S}.
This serves  as an answer to the question posed by Feh\'{e}r,
N\'{e}methi and Rim\'{a}nyi.

Let us proceed with the tableau interpretation of  factorial Grothendieck
polynomials.
Let $\lambda=(\lambda_1,\lambda_2,\ldots, \lambda_\ell)$ be an integer
partition, that is, $\lambda_1,\lambda_2,\ldots,\lambda_\ell$ are
nonnegative integers such that
$\lambda_1\geq \lambda_2\geq \cdots\geq \lambda_\ell\geq 0$.
Write $|\lambda|=\lambda_1+\lambda_2+\cdots+\lambda_\ell$.
The Young diagram of $\lambda$ is a left-justified array
of squares with $\lambda_i$ squares in row $i$. A square $\alpha$ of $\lambda$
in row $i$ and column $j$ is denoted $\alpha=(i,j)$.
If  no confusion arises,
we do not distinguish a partition and its Young diagram.
A set-valued   tableau $T$ of shape $\lambda$
is an assignment of finite sets of positive integers into
the squares of $\lambda$ such that the
sets in each row (respectively, column) are
weakly (respectively, strictly) increasing,
where, for two finite sets $A$ and $B$
 of positive integers, we write $A\leq B$ if $\max A \leq  \min B$
and $A< B$ if $\max A < \min B$.
The notion of set-valued tableaux was introduced by
Buch \cite{Buch} in his study of  the Littlewood-Richardson
rule for stable Grothendieck polynomials.
Let $T(\alpha)$   denote the subset filled in a square $\alpha$.
We write  $|T|=\sum_{\alpha\in \lambda} |T(\alpha)|$ and
write $c(\alpha)=j-i$ for the content of    $\alpha=(i,j)$.

 Let $\mathcal{T}{(\lambda, n)}$ denote
 the set of  set-valued tableaux $T$ of shape $\lambda$ such that
each subset appearing in $T$ is a subset of $[n]$.
For variables $\beta$, $x_i$ and $y_j$, we adopt the following notation
as used by Fomin and Kirillov \cite{FK}:
\[x_i\oplus y_j=x_i+y_j+\beta x_iy_j.\]
The  factorial Grothendieck polynomial  $G_\lambda(x|y)$ is defined as
\begin{equation}\label{def fac combi}
G_{\lambda}(x|y)=\sum_{T\in \mathcal{T}{(\lambda,n)}}
\beta^{|T|-|\lambda|}\prod_{\alpha\in T}
\prod_{t\in T(\alpha)}\left(x_{t}\oplus y_{t+c(\alpha)}\right).
\end{equation}
In the case $\beta=0$, $G_{\lambda}(x|y)$ becomes the
factorial Schur function $s_{\lambda}(x|y)$,
while in the case $\beta=0$ and $y=0$, $G_{\lambda}(x|y)$ specializes
to the Schur function $s_\lambda(x)$.

For example,
there are three set-valued tableaux in $\mathcal{T}{(\lambda,n)}$
for  $\lambda=(2,1)$ and $n=2$:
\[
\begin{ytableau}
\scriptstyle1 &  \scriptstyle1\\
\scriptstyle2
\end{ytableau}
\qquad
\begin{ytableau}
\scriptstyle1 &  \scriptstyle2\\
\scriptstyle2
\end{ytableau}
\qquad
\begin{ytableau}
\scriptstyle1 &  \scriptstyle12\\
\scriptstyle2
\end{ytableau}.
\]
By  \eqref{def fac combi}, we see that
\begin{align}\label{def_ex}
G_{\lambda}(x|y)=&(x_1\oplus y_1)\cdot(x_1\oplus y_2)\cdot(x_2\oplus y_1)+(x_1\oplus y_1)\cdot(x_2\oplus y_3)\cdot(x_2\oplus y_1)\notag\\[5pt]
&\ \ \ +\beta\cdot(x_1\oplus y_1)\cdot(x_1\oplus y_2)\cdot(x_2\oplus y_3)\cdot(x_2\oplus y_1).
\end{align}

Factorial Grothendieck polynomials are double Grothendieck
polynomials of Grassmannian permutations.
The double Grothendieck polynomials were introduced by
Lascoux and Sch\"utzenberger  \cite{La-Sc2}
as  polynomial representatives of the equivariant $K$-theory classes of
structure sheaves of Schubert varieties in the flag manifold.
Let $S_p$ denote the symmetric group of permutations of $[p]$.
For a permutation $w\in S_p$, the double Grothendieck polynomial $\mathfrak{G}_w(x;y)$
is defined  based on the isobaric divided difference operator $\pi_i$ acting
  on $\mathbb{Z}[\beta][x,y]$. For  $f\in \mathbb{Z}[\beta][x,y]$,
  let $s_i f$ be  obtained from $f$ by interchanging $x_i$
and $x_{i+1}$. Then
\[\pi_i\, f=\frac{(1+\beta x_{i+1})f-(1+\beta x_i)s_i f}{x_i-x_{i+1}}.\]
The length of $w$,
denoted $\ell(w)$, is equal to the number of pairs $(i,j)$ such that $1\leq i<j\leq p$
and $w_i>w_j$. For the longest permutation $w_0=p\cdots 2 1$, set
\[\mathfrak{G}_{w_0}(x;y) =\prod_{i+j\leq p}(x_i\oplus y_j).\]
If $w\neq w_0$, then one can choose a simple transposition $s_i$ such that
$\ell(ws_i)>\ell(w)$. Set
\[\mathfrak{G}_w(x;y) =\pi_i\, \mathfrak{G}_{ws_i}(x;y).\]
 Note that the above defined double Grothendieck polynomials  are called  double $\beta$-polynomials \cite{FK}, and
reduce  to the ordinary double Grothendieck polynomials  in the case $\beta=-1$.
In the case $\beta=0$,   $\mathfrak{G}_w(x;y)$ equals  the double
Schubert polynomial $\mathfrak{S}_w(x;y)$ \cite{La-Sc1,Mac4}.

A permutation $w\in S_p$ is a  Grassmannian permutation  if
there is at most one position, say $n$, such that $w_n>w_{n+1}$.
To a Grassmannian permutation  $w$, one can associate a
partition $\lambda(w)=(\lambda_1 , \ldots, \lambda_n)$ where $\lambda_i=w_{n-i+1}-(n-i+1)$. For a Grassmannian
permutation $w$ with $\lambda(w)=\lambda$, it has been shown that
$\mathfrak{G}_w(x;y)=G_\lambda(x|y)$,  see for example \cite{Buch,KMY,KMY-2,Mc,Mat17}.

On the other hand, there have been determinantal formulas for factorial
Grothendieck polynomials \cite{HIMN,Hu-Ma,Ik-Na,Mat}.
For the purpose of this paper, we need the following
determinantal formula for $ G_\lambda(x|y)$ due to Ikeda and Naruse   \cite{Ik-Na}:
\begin{equation}\label{frac deter}
  G_\lambda(x|y)=\frac{\det\left([x_i|y]^{\lambda_j+n-j}(1+\beta x_i)^{j-1}\right)_{1\leq i, j\leq n}}
{\prod\limits_{1\leq i<j\leq n}(x_i-x_j)},
\end{equation}
where
\[[x_i|y]^j=(x_i\oplus y_1)(x_i\oplus y_2)\cdots (x_i\oplus y_j).\]
For example, for $\lambda=(2,1)$ and $n=2$, by  \eqref{frac deter}, we  have
\begin{equation*}\label{def_ex1}
  G_{\lambda}(x|y)=\frac{1}{x_1-x_2}
    \begin{vmatrix}
    [x_1|y]^3&[x_1|y](1+\beta x_1)\\[5pt]
    [x_2|y]^3&[x_2|y](1+\beta x_2)
  \end{vmatrix},
\end{equation*}
which is the same as  \eqref{def_ex} via a simple calculation.

\section{The Gustafson-Milne type identity}\label{Sect2}

In this section, we prove a Gustafson-Milne type identity for  factorial
Grothendieck polynomials which can be  stated as follows.

\begin{thm}\label{thm gen milne}
Let
$x=(x_1,\ldots,x_n)$ and
$y=(y_1,y_2,\ldots)$ be two sets of variables.
For a partition
$\lambda=(\lambda_1,\lambda_2,\ldots,\lambda_k)$, we have
\begin{equation}\label{XXX}
  G_{(\lambda_1-n+k,\lambda_2-n+k,\ldots,\lambda_k-n+k)}(x|y)=\sum_{S\in\binom{[n]}{k}}
  G_{(\lambda_1,\lambda_2,\ldots,\lambda_k)}(x_S|y)\frac{\prod\limits_{j\in \overline{S}}(1+\beta x_j)^k}{\prod\limits_{i\in S} \prod\limits_{j\in \overline{S}}(x_i-x_j)},
  \end{equation}
  where the left-hand side of \eqref{XXX} is zero in the case  $\lambda_k-n+k<0$.
\end{thm}

Theorem \ref{thm gen milne} becomes an
identity on factorial Schur functions in the case $\beta=0$, and
specializes to the Gustafson-Milne identity \eqref{thm milne} in
 the case $\beta=0$ and $y=0$.

Taking $\lambda=(n-1)$  and $\lambda=(m)$ respectively in Theorem \ref{thm gen milne}, we obtain   two identities  which contain the Good's identity  \eqref{Good} and the Louck's identity \eqref{thm louck eq} as special cases.

\begin{cor}
  Let
$x=(x_1,\ldots,x_n)$ and
$y=(y_1,y_2,\ldots)$ be two sets of variables. Then,
  \begin{equation}\label{FII}
  1=\sum_{i=1}^n [x_i|y]^{n-1}\prod_{j=1\atop{j\not= i}}^n\frac{1+\beta x_j}{x_i-x_j}
  \end{equation}
  and
  \begin{equation*}
  h_{m-n+1}(x|y)=\sum_{i=1}^n[x_i|y]^m\prod\limits_{j=1\atop{j\not=i}}^n\frac{1+\beta x_j}{x_i-x_j}.
\end{equation*}
\end{cor}

The following lemma is key to the proof of Theorem \ref{thm gen milne},
 which can be viewed as a generalization of the  Vandermonde determinant.

\begin{lem}\label{lem det}
For $n\geq 1$,
\begin{equation}\label{eq lem det}
  \det\left([x_r|y]^{n-c}(1+\beta x_r)^{c-1}\right)_{1\leq r,c\leq n}=\prod_{1\leq i<j\leq n}(x_i-x_j),
 \end{equation}
  where, in the case $n=1$, both sides of \eqref{eq lem det} are equal to one.
\end{lem}
\begin{proof}
For a set  $S=\{i_1<i_2<\cdots<i_k\}$ of positive integers, let
\[y_S=y_{i_1}y_{i_2}\cdots y_{i_k}\]
 and
\[(1+\beta y)_S=(1+\beta y_{i_1})(1+\beta y_{i_2})\cdots (1+\beta y_{i_k}).\]
With the above notation, we define
\[E^{(\beta)}_k(Y_n)=\sum_{S\in\binom{[n]}{k}}y_{S}(1+\beta y)_{\overline{S}},\]
where $\overline{S}=[n]\setminus S$. Here,  $E^{(\beta)}_k(Y_n)=0$ unless   $0\leq k\leq n$.
Note that when $\beta=0$, the polynomial  $E^{(\beta)}_k(Y_n)$ is the elementary symmetric function  $e_k(y_1,y_2,\ldots,y_n)$.
It is easily checked  that $E^{(\beta)}_k(Y_n)$
satisfies the following recurrence relation:
\begin{equation}\label{YUN-2}
E^{(\beta)}_k(Y_n)=(1+\beta y_n)E^{(\beta)}_{k}(Y_{n-1})+y_nE^{(\beta)}_{k-1}(Y_{n-1}).
\end{equation}

The entry on the left-hand side of  \eqref{eq lem det} can be reformulated as
\begin{align*}
[x_r|y]^{n-c}(1+\beta x_r)^{c-1}&=\left(\sum_{h=0}^{n-c} E^{\beta}_h (Y_{n-c})x_r^{n-c-h}\right)
\left(\sum_{i=0}^{c-1} {c-1 \choose i} \beta^i x_r^i\right)\nonumber\\[5pt]
&=\sum_{j=0}^{n-1}\left(\sum_{i=0}^{j}\binom{c-1}{i}E^{(\beta)}_{n-c-j+i}
(Y_{n-c})\beta^i\right)x_r^j\nonumber\\[5pt]
&=\sum_{j=1}^{n }x_r^{n-j}\left(\sum_{i=0}^{n-j}\binom{c-1}{i}E^{(\beta)}_{i-c+j}
(Y_{n-c})\beta^i\right),
\end{align*}
which implies that the left-hand side of  \eqref{eq lem det} can be  written as
a product involving a  Vandermonde determinant:
\begin{align*}
   &\det\left([x_r|y]^{n-c}(1+\beta x_r)^{c-1}\right)_{1\leq r,c\leq n}\nonumber\\[5pt]
   &\ \ \ =\det(x_r^{n-c})_{1\leq r,c\leq n}\cdot \det\left(\sum_{i=0}^{n-r}\binom{c-1}{i}E^{(\beta)}_{i-c+r}(Y_{n-c})\beta^i\right)_{1\leq r,c\leq n}\nonumber\\[5pt]
   &\ \ \ =\prod_{1\leq i<j\leq n}(x_i-x_j)\cdot \det(H(n)).
\end{align*}

It remains to evaluate that $\det(H(n))=1$.
Denote by $\mathrm{Col}_c$ the $c$-th column of  $H(n)$.
 For  $1\leq c \leq n-1$,  we apply the  column transformation
  $\mathrm{Col}_c-y_{n-c}\, \mathrm{Col}_{c+1}$, that is, $\mathrm{Col}_c$
 is replaced by $\mathrm{Col}_c-y_{n-c}\,\mathrm{Col}_{c+1}$.
Let $H'(n)$ be the resulting matrix after such column transformations.
Then the entry of $H'(n)$ in row $1\leq r\leq n$ and column $1\leq c\leq n-1$ can be simplified as
\begin{align}\label{H(n)tra}
  &\sum_{i=0}^{n-r}\binom{c-1}{i}E^{(\beta)}_{i-c+r}(Y_{n-c})\beta^i
  -y_{n-c}\sum_{i=0}^{n-r}\binom{c}{i}E^{(\beta)}_{i-c+r-1}(Y_{n-c-1})\beta^i\nonumber\\[5pt]
  &=\sum_{i=0}^{n-r}\binom{c-1}{i}E^{(\beta)}_{i-c+r}(Y_{n-c})\beta^i
  -y_{n-c}\sum_{i=0}^{n-r}\left[\binom{c-1}{i}+\binom{c-1}{i-1}\right]
  E^{(\beta)}_{i-c+r-1}(Y_{n-c-1})\beta^i\nonumber\\[5pt]
  &=\sum_{i=0}^{n-r}\binom{c-1}{i}\left[E^{(\beta)}_{i-c+r}(Y_{n-c})-y_{n-c}
  E^{(\beta)}_{i-c+r-1}(Y_{n-c-1})\right]\beta^i\nonumber\\[5pt]
  &\ \ \ \ \ \ \ -y_{n-c}\sum_{i=0}^{n-r}\binom{c-1}{i-1}E^{(\beta)}_{i-c+r-1}(Y_{n-c-1})
  \beta^{i}\nonumber\\[5pt]
  &=(1+\beta y_{n-c})\sum_{i=0}^{n-r-1}\binom{c-1}{i}E^{(\beta)}_{i-c+r}(Y_{n-c-1})\beta^i -y_{n-c}\sum_{i=0}^{n-r-1}\binom{c-1}{i}E^{(\beta)}_{i-c+r}(Y_{n-c-1})\beta^{i+1}\nonumber\\[5pt] &=\sum_{i=0}^{n-r-1}\binom{c-1}{i}E^{(\beta)}_{i-c+r}(Y_{n-c-1})\beta^i,
\end{align}
where the third equality follows from \eqref{YUN-2}.

Notice that   the first $n-1$ entries in the
$n$-th row of $H'(n)$  are all zero, while that last entry in the $n$-th row of
$H'(n)$ is one. Therefore, by \eqref{H(n)tra} and using the Laplace expansion along the $n$-th row of $H'(n)$, we have
\begin{align*}
 \det(H(n))&=\det(H'(n))=\det\left(\sum_{i=0}^{n-r-1}\binom{c-1}{i}E^{(\beta)}_{i-c+r}(Y_{n-c-1})\beta^i\right)_{1\leq r,c\leq n-1}\\[5pt]
  &=\det(H(n-1)).
\end{align*}
With the initial value  $\det(H(1))=1$ and using induction, we deduce that
$\det(H(n))=1$ for  $n\geq 1$. This completes the proof.
\end{proof}

Now we are ready to present a proof of  Theorem \ref{thm gen milne}.

\begin{proof}[Proof of Theorem \ref{thm gen milne}.]
Write $\nu=(\lambda_1-n+k,\lambda_2-n+k,\ldots,\lambda_k-n+k)$.
By   \eqref{frac deter},
\begin{align}\label{eq pf milne}
G_{\nu}(x|y)
=\frac{\det\left([x_i|y]^{\nu_j+n-j}(1+\beta x_i)^{j-1}\right)_{1\leq i, j\leq n}}
{\prod\limits_{1\leq i<j\leq n}(x_i-x_j)},
\end{align}
where $\nu_j=0$ for $j>k$.
By the Laplace expansion of a determinant
 along  the first $k$ columns,
the numerator on  the right-hand side of \eqref{eq pf milne}
equals
\begin{align}
&\det\left([x_i|y]^{\nu_j+n-j}(1+\beta x_i)^{j-1}\right)_{1\leq i, j\leq n}\notag\\[5pt]
&\ \ \ \ \ \ =\sum_{S\in \binom{[n]}{k}}(-1)^{\binom{k+1}{2}+\sum_{r\in S}r}\det\left([x_r|y]^{\lambda_c+k-c}(1+\beta x_r)^{c-1}\right)_{r\in S\atop{1\leq c\leq k}}\notag\\[5pt]
&\ \ \ \ \ \ \ \ \ \ \ \ \ \ \cdot\det\left([x_r|y]^{n-k-c}(1+\beta x_r)^{k+c-1}\right)_{r\in \overline{S}\atop{1\leq c\leq n-k}}\notag\\[5pt]
&\ \ \ \ \ \ =\sum_{S\in \binom{[n]}{k}}(-1)^{\binom{k+1}{2}+\sum_{r\in S}r}\det\left([x_r|y]^{\lambda_c+k-c}(1+\beta x_r)^{c-1}\right)_{r\in S\atop{1\leq c\leq k}}\notag\\[5pt]
&\ \ \ \ \ \ \ \ \ \ \ \ \ \ \cdot\det\left([x_r|y]^{n-k-c}(1+\beta x_r)^{c-1}\right)_{r\in \overline{S}\atop{1\leq c\leq n-k}}\cdot
\prod_{j\in \overline{S}}(1+\beta x_j)^{k}.
\label{milne numerator}
\end{align}
By Lemma \ref{lem det}, we see that
\begin{align*}
\det\left([x_r|y]^{n-k-c}(1+\beta x_r)^{c-1}\right)_{r\in \overline{S}\atop{1\leq c\leq n-k}}=\prod_{i<j\atop{i,j\in\overline{S}}}(x_i-x_j),
\end{align*}
and so \eqref{milne numerator} can be expressed as
\begin{align}\label{READ}
&\det\left([x_i|y]^{\nu_j+n-j}(1+\beta x_i)^{j-1}\right)_{1\leq i, j\leq n}\notag\\[5pt]
&\ \ \ \ \ \ =\sum_{S\in \binom{[n]}{k}}(-1)^{\binom{k+1}{2}+\sum_{r\in S}r}\det\left([x_r|y]^{\lambda_c+k-c}(1+\beta x_r)^{c-1}\right)_{r\in S\atop{1\leq c\leq k}}\notag\\[5pt]
&\ \ \ \ \ \ \ \ \ \ \ \ \ \ \cdot\prod_{i<j\atop{i,j\in\overline{S}}}(x_i-x_j)\cdot \prod_{j\in \overline{S}}(1+\beta x_j)^{k}.
\end{align}

For any $k$-subset $S\subseteq [n]$,
notice that
\begin{align*}
  \prod\limits_{{i\in S\atop{j\in \overline{S}}}}(x_i-x_j)
    & =
   \prod\limits_{{i<j}\atop{{i\in S\atop{j\in\overline{S}}}}}(x_i-x_j)
   \prod\limits_{{i>j}\atop{{i\in S\atop{j\in \overline{S}}}}}(x_i-x_j)\\[5pt]
   &=(-1)^{-\binom{k+1}{2}+\sum_{r\in S}r}
   \prod\limits_{{i<j}\atop{{i\in S\atop{j\in\overline{S}}}}}(x_i-x_j)
   \prod\limits_{{i>j}\atop{{i\in S\atop{j\in \overline{S}}}}}(x_j-x_i).
\end{align*}
Hence the
 denominator on the right-hand side of   \eqref{eq pf milne} can be rewritten as
\begin{align}
 \prod\limits_{1\leq i<j\leq n}(x_i-x_j)
 =&\prod\limits_{{i<j\atop{i,j\in S}}}(x_i-x_j)
  \prod\limits_{{i<j\atop{i,j\in\overline{S}}}}(x_i-x_j)
  \prod\limits_{{i<j}\atop{{i\in S\atop{j\in \overline{S}}}}}(x_i-x_j)
  \prod\limits_{{i>j}\atop{{i\in S\atop{j\in \overline{S}}}}}(x_j-x_i)\notag\\[5pt]
 =&(-1)^{-\binom{k+1}{2}+\sum_{r\in S}r}\prod\limits_{{i<j\atop{i,j\in S}}}(x_i-x_j)
 \prod\limits_{{i<j\atop{i,j\in\overline{S}}}}(x_i-x_j)\prod\limits_{{i\in S\atop{j\in \overline{S}}}}(x_i-x_j).\label{milne denumerator}
\end{align}

Putting  \eqref{READ} and   \eqref{milne denumerator} into \eqref{eq pf milne}, we deduce that
\begin{align}\label{ABC}
   &G_{(\lambda_1-n+k,\lambda_2-n+k,\ldots,\lambda_k-n+k)}(x|y)\nonumber\\[5pt]
   &\ \ \ \ \ =\sum_{S\in \binom{[n]}{k}}\frac{\det\left([x_r|y]^{\lambda_c+k-c}(1+\beta x_r)^{c-1}\right)_{r\in S\atop{1\leq c\leq k}}
   \prod\limits_{j\in \overline{S}}(1+\beta x_j)^{k}}{\prod\limits_{{i<j\atop{i,j\in S}}}(x_i-x_j)
    \prod\limits_{{i\in S\atop{j\in \overline{S}}}}(x_i-x_j)}.
\end{align}
By \eqref{frac deter},
\[ G_{(\lambda_1,\lambda_2,\ldots,\lambda_k)}(x_S|y)=
\frac{\det\left([x_r|y]^{\lambda_c+k-c}(1+\beta x_r)^{c-1}\right)_{r\in S\atop{1\leq c\leq k}}}{\prod\limits_{{i<j\atop{i,j\in S}}}(x_i-x_j)},\]
which together with \eqref{ABC} yields \eqref{XXX}.
This completes the proof.
\end{proof}

\section{The  Feh\'{e}r-N\'{e}methi-Rim\'{a}nyi type identity}\label{sec pf fbr12}

In this section, we provide a
Feh\'{e}r-N\'{e}methi-Rim\'{a}nyi type identity
on  factorial Grothendieck polynomials.
The proof also relies  on \eqref{frac deter} and Lemma \ref{lem det}.
Restricting to  Schur functions, we obtain a
combinatorial proof of the identity \eqref{FNR-S}
 given by Feh\'{e}r, N\'{e}methi and Rim\'{a}nyi.

\begin{thm}\label{combinat_thm}
Let
$x=(x_1,\ldots,x_n)$ and
$y=(y_1,y_2,\ldots)$ be two sets of variables.
Let $\lambda=(\lambda_1,\lambda_2,\ldots,\lambda_k)$
be a partition such that  $\lambda_1\leq m-k$, and
let \[\mu=({\small{\underbrace{m-k,m-k,\ldots,m-k}_{n-k}}},
\lambda_1,\lambda_2,\ldots,\lambda_k).\]
Then we have
\begin{equation}\label{eq thm'}
   G_{\mu}
   (x|y)=\sum_{S\in\binom{[n]}{k}} G_{\lambda}(x_S|y)\frac{ \prod\limits_{i\in S}(1+\beta x_i)^{n-k}\prod\limits_{j\in \overline{S}} [x_j|y]^m}
{ \prod\limits_{i\in S}\prod\limits_{j\in \overline{S}} (x_j-x_i)}.
\end{equation}
\end{thm}

Note that in the case $\beta=0$ and $y_i=0$, Theorem
\ref{combinat_thm} specializes to \eqref{FNR-S}.

\begin{proof}[Proof of Theorem \ref{combinat_thm}.]
Again, by\eqref{frac deter}, we have
\begin{equation}\label{frac grothendieck mu}
G_\mu(x|y)=\frac{\det\left([x_i|y]^{\mu_j+n-j}(1+\beta x_i)^{j-1}\right)_{1\leq i, j\leq n}}
{\prod\limits_{1\leq i<j\leq n}(x_i-x_j)}.
\end{equation}
Using the Laplace expansion by the last $k$ columns,
the numerator in \eqref{frac grothendieck mu} can be written as
\begin{align}\label{numerator1}
&\sum_{S\in \binom{[n]}{k}}(-1)^{nk-\binom{k}{2}+\sum_{r\in S}r}\det\left([x_r|y]^{\lambda_c+k-c}(1+\beta x_r)^{n-k+c-1}\right)_{r\in S\atop{1\leq c\leq k}}\notag\\
&\ \ \ \ \ \cdot\det\left([x_r|y]^{n-k+m-c}(1+\beta x_r)^{c-1}\right)_{r\in \overline{S}\atop{1\leq c\leq n-k}}\notag\\[5pt]
&=\sum_{S\in \binom{[n]}{k}}(-1)^{nk-\binom{k}{2}+\sum_{r\in S}r}\det\left([x_r|y]^{\lambda_c+k-c}(1+\beta x_r)^{c-1}\right)_{r\in S\atop{1\leq c\leq k}}
\prod_{i\in S}(1+\beta x_i)^{n-k}\notag\\
&\ \ \ \ \ \ \cdot\det\left([x_r|y]^{n-k+m-c}(1+\beta x_r)^{c-1}\right)_{r\in \overline{S}\atop{1\leq c\leq n-k}}.
\end{align}

Let us first evaluate  the last factor in \eqref{numerator1}.
It is easy to see that
\begin{align*}
  \det\left([x_r|y]^{n-k+m-c}(1+\beta x_r)^{c-1}\right)_{r\in \overline{S}\atop{1\leq c\leq n-k}}
  =\prod_{j\in \overline{S}}[x_j|y]^m \cdot
  \det\left([x_r|y]_m^{n-k-c}(1+\beta x_r)^{c-1}\right)_{r\in \overline{S}\atop{1\leq c\leq n-k}},
\end{align*}
where, for $p\geq 0$ and $q\geq 0$,
\[[x_i|y]^q_p=(x_i\oplus y_{p+1})\cdots (x_i\oplus y_{p+q}).\]
Note that
\[\det\left([x_r|y]_m^{n-k-c}(1+\beta x_r)^{c-1}\right)_{r\in \overline{S}\atop{1\leq c\leq n-k}}\]
is obtained from
\begin{equation}\label{PQW}
\det\left([x_r|y]^{n-k-c}(1+\beta x_r)^{c-1}\right)_{r\in \overline{S}\atop{1\leq c\leq n-k}}
\end{equation}
by replacing $y_i$ with $y_{i+m}$ for $i\geq 1$.
By Lemma \ref{lem det}, \eqref{PQW} is  equal to \[\prod_{i<j\atop{i,j\in\overline{S}}}(x_i-x_j),\]
which is independent of the variables $y_i$,
and so we have
\begin{equation*}
\det\left([x_r|y]_m^{n-k-c}(1+\beta x_r)^{c-1}\right)_{r\in \overline{S}\atop{1\leq c\leq n-k}}=\prod_{i<j\atop{i,j\in\overline{S}}}(x_i-x_j).
\end{equation*}
Hence we obtain that
\begin{align}\label{GSA}
  \det\left([x_r|y]^{n-k+m-c}(1+\beta x_r)^{c-1}\right)_{r\in \overline{S}\atop{1\leq c\leq n-k}}
  =\prod_{j\in \overline{S}}[x_j|y]^m \cdot \prod_{i<j\atop{i,j\in\overline{S}}}(x_i-x_j).
\end{align}

As to the denominator in \eqref{frac grothendieck mu},
by \eqref{milne denumerator}, it follows that
\begin{align}
 \prod\limits_{1\leq i<j\leq n}(x_i-x_j)
 =&(-1)^{-\binom{k+1}{2}+\sum_{r\in S}r}\prod\limits_{{i<j\atop{i,j\in S}}}(x_i-x_j)
 \prod\limits_{{i<j\atop{i,j\in\overline{S}}}}(x_i-x_j)\prod\limits_{{i\in S\atop{j\in \overline{S}}}}(x_i-x_j)\notag\\[5pt]
  =&(-1)^{k(n-k)-\binom{k+1}{2}+\sum_{r\in S}r}\prod\limits_{{i<j\atop{i,j\in S}}}(x_i-x_j)
 \prod\limits_{{i<j\atop{i,j\in\overline{S}}}}(x_i-x_j)\prod\limits_{{i\in S\atop{j\in \overline{S}}}}(x_j-x_i)\notag\\[5pt]
 =&(-1)^{k(n-k)+\binom{k+1}{2}+\sum_{r\in S}r}\prod\limits_{{i<j\atop{i,j\in S}}}(x_i-x_j)
 \prod\limits_{{i<j\atop{i,j\in\overline{S}}}}(x_i-x_j)\prod\limits_{{i\in S\atop{j\in \overline{S}}}}(x_j-x_i)\notag\\[5pt]
  =&(-1)^{nk-\binom{k}{2}+\sum_{r\in S}r}\prod\limits_{{i<j\atop{i,j\in S}}}(x_i-x_j)
 \prod\limits_{{i<j\atop{i,j\in\overline{S}}}}(x_i-x_j)\prod\limits_{{i\in S\atop{j\in \overline{S}}}}(x_j-x_i).\label{denumerator}
\end{align}

Finally, notice that
\begin{equation}\label{FF}
G_\lambda(x_S|y)=
\frac{\det\left([x_r|y]^{\lambda_c+k-c}(1+\beta x_r)^{c-1}\right)_{r\in S\atop{1\leq c\leq k}}}
{\prod\limits_{{i<j\atop{i,j\in S}}}(x_i-x_j)}.
\end{equation}
Combining  \eqref{frac grothendieck mu}, \eqref{numerator1}, \eqref{GSA}, \eqref{denumerator} and \eqref{FF}, we arrive at \eqref{eq thm'},
as desired.
\end{proof}

Taking  $\lambda_1=\lambda_2=\cdots=\lambda_k=0$ and $m=k$ in Theorem \ref{combinat_thm}, we are led to
 another  generalization of the Good's identity.

\begin{cor}\label{cor good gern}
Let
$x=(x_1,\ldots,x_n)$ and
$y=(y_1,y_2,\ldots)$ be two sets of variables.  Then,
for $0\leq k\leq n$,  we have
  \begin{equation*}
  1=\sum_{S\in\binom{[n]}{k}}
   \frac{ \prod\limits_{i\in S}(1+\beta x_i)^{n-k}\prod\limits_{j\in \overline{S}} [x_j|y]^k}
{\prod\limits_{i\in S}\prod\limits_{j\in \overline{S}} (x_j-x_i)}.
  \end{equation*}
\end{cor}

Note that Corollary \ref{cor good gern} specializes to
 identity \eqref{FII} in the case $k=n-1$.

\vspace{0.5cm}
 \noindent{\bf Acknowledgments.}
This work was supported by  the 973
Project and the National Science Foundation of China.

\end{document}